\documentclass[preprint]{imsart}

\RequirePackage{amsthm,amsmath,amsfonts,amssymb}
\RequirePackage[numbers]{natbib}

\startlocaldefs
\usepackage[T1]{fontenc}
\usepackage{graphicx,color}
\usepackage{amsmath,amsthm,amssymb}
\newtheorem{thm}{Theorem}[section]
\newtheorem{prop}[thm]{Proposition}
\newtheorem{lemma}[thm]{Lemma}
\newtheorem{cor}[thm]{Corollary}
\theoremstyle{definition}

\theoremstyle{remark}
\newtheorem{rmk}{Remark}[section]
\def\1{\mbox{1\hspace{-.35em}1}} 

\def\R{\mathbb{R}}

\def\N{\mathbb{N}}
\def\L{\mathbb{L}}
\def\P{\mathbb{P}}
\def\E{\mathbb{E}}
\def\Z{\mathbb{Z}}
\def\L{\mathbb{L}}

\def\lip{\mbox{Lip\,}}
\def\Lip{\mbox{Lip\,}}
\def\cov{\mbox{Cov}}

\def\v{\mbox{Var}\,}
\def\cov{\mbox{Cov}\,}

	\newcommand{\rmi}{{\rm (i) $\hspace{1mm}$}}

\newcommand{\rmii}{{\rm (ii) $\hspace{1mm}$}}
\newcommand{\rmiii}{{\rm (iii)$\hspace{1mm}$}}

\endlocaldefs
\usepackage[margin=2.25cm]{geometry}
\begin{document}

\begin{frontmatter}
\title{A Dynamic Taylor's Law}

\begin{aug}
\author[A]{\fnms{Victor} \snm{De la Pena}\ead[label=e1]{vhdl@columbia.edu}},
\author[B]{\fnms{Paul} \snm{Doukhan}\ead[label=e2]{doukhan@cyu.fr}}
\and
\author[C]{\fnms{Yahia} \snm{Salhi}\ead[label=e3]{yahia.salhi@univ-lyon1.fr}}

\address[A]{Columbia University, New York, USA, \printead{e1}}

\address[B]{CY University, AGM UMR 8088, site Saint-Martin,
	95000 Cergy-Pontoise, France, \printead{e2}}

\address[C]{University of  Lyon, UCBL, ISFA LSAF EA2429, F-69007, Lyon, France, \printead{e3}}

\end{aug}

\begin{abstract}
Taylor's power law (or fluctuation scaling) states that on comparable populations, the variance of each sample is approximately proportional to a power of the mean of the population. It has been shown to hold by empirical observations in a broad class of disciplines including demography, biology, economics, physics and mathematics. 
 In particular, it has been observed in the problems involving population dynamics, market trading, thermodynamics and number theory. 
 For this many authors consider panel data in order to obtain laws of large numbers and the possibility to fit those expressions; essentially we aim at considering ergodic behaviors without independence. Thus we restrict the study to stationary time series and we develop different Taylor exponents in this setting.
 From a theoretic point of view, there has been a growing interest on the study of the behavior of such a phenomenon. Most of these works focused on the so-called static Taylor related to independent samples. In this paper, we introduce a dynamic Taylor's law for dependent samples using self-normalised expressions involving Bernstein blocks. A central limit theorem (CLT) is proved under either weak dependence or strong mixing assumptions for the marginal process. The limit behavior of such a new index involves the series of covariances unlike the classic framework where the limit behavior involves the marginal variance. We also provide an asymptotic result for for a goodness-of-fit testing suited to check whether the corresponding dynamical Taylor's law holds in empirical studies. Moreover, we also obtain a consistent estimation of the Taylor's exponent.

\end{abstract}



\end{frontmatter}
\section{Introduction}

An important criterion used to describe the dynamic of populations is exhibited in \cite{cohen2012taylor}, among others, through the expression known as Taylor's laws. This originated as an empirical pattern in ecology in such a way that, on comparable populations, the variance of each sample was approximately proportional to a power of the mean of that sample. Thousands of papers have been dedicated to the study of Taylor's Law. This limits our ability to provide a comprehensive review. An important survey on the topic is \cite{eisler2008fluctuation}. A key motivation for our work can be found in \cite{kendal2011tweedie} which provides a central-limit-like convergence that explains Taylor's Law (TL) as well as \cite{brown2017taylor} which introduces a self-normalized empirical version of Taylor's law for some distributions with infinite mean.


The question set here is about what happens in case one observes only one trajectory of a random phenomenon. We clearly need ergodicity conditions to consistently investigate the expressions related with TL.
We thus  develop a theory for TL under dependence in the case when only a trajectory of the process of interest $(X_t)_{t\in\Z}$ is observed. It can be of course accommodated to the context of independent copies of the process $(X_t)_{t\in \Z}$ observed over different samples. This is, for instance, the case when dealing with mortality time series over ages or even regions \cite{bohk2015taylor}, which means that the dependence in the sample has already been considered.
In order to obtain laws of large numbers and the possibility to fit those expressions we restrict our paper to a specific frame. Essentially, we aim at considering ergodic behaviors and focus on stationary time series.  In this case we are in position to define general TL possibly taking into account the dynamic behaviour of the process of interest and not only its marginal distribution. 

To achieve this we proceed in two steps. Firstly we extend strictly the TL to the ergodic (dependent setting) which we will call a static TL since it only relies on the marginal distribution of the stationary process $(X_t)_{t\in \Z}$.  Secondly, we introduce an alternative dynamic TL. The main input of the paper is to introduce such a TL which does not only involve the marginal distribution of $X_0$ but instead, relies on the whole second order structure of the process $(X_t)_{t\in \Z}$, and thus accounts for the dependence of the blocks.
Recall, for example, that the sample average and variance are accurate measures of the mean and variation in the population which gives sense to TL along some trajectory.
However, the timeliness of our approach is supported by the findings of \cite{RLSRC} where it is shown that changes in synchrony (that may be caused by climate change) modifies and can invalidate the TL.  By incorporating the entire history of the time series, our dynamic approach to Taylor's law may mitigate the effect of these changes.

The current work is developed in the context of weakly dependent variables which include for example Bernoulli shifts of i.i.d. r.v.'s such as $X_t = F(X_{t-1},X_{t-2}, \ldots; \xi_t)$ \cite{doukhan2008weakly}, which depends on the complete past history of the process and some innovations; the simplest model of it is the case of ergodic Markov chains $X_t= F (X_{t-1} ; \xi_t )$ but also infinite moving averages of iid inputs $X_t=\sum_{j=0}^{\infty}a_j X_{t-j}+\xi_t$ are such models. Larges classes of examples including ARCH, GARCH-type models or integer valued GLM models possibly integer valued may be found both in \cite{D1994} and in \cite{DDLLLP}, see also e.g. \cite{DFR}, \cite{FSTD} and \cite{DN} for more models. Note that those classes of models and many others may easily seen to fit the conditions in the current results, which gives sense to our settings and to our  results. Such models include dependence over the time and should rather be used in order to describe dynamical evolution of a population \cite{cohen2012taylor}. It is, in fact, relevant in populations dynamics and in particular for ecological application, see also \cite{cohen2016population,saitoheffects} and the reference therein. 

The static Taylor's law is a proper characteristic of marginal distributions. Since we consider dynamical issues we think that a  Taylor's law depending  on the whole distribution of the analysed process  is more adapted. In \cite{cohen2012taylor}, the question of checking the validity for Taylor's law for some random phenomenon is addressed in the setting of i.i.d. sequences. Formally, for an integer $k>1$ and a sequence of positive random variables, the relation $\v Y=c (\E Y)^\alpha$ with $c\in \R$ and $\alpha>0$ was shown to hold in many empirical applications. This means that the consistency of empirical counterparts for those expressions is proved for the convergence in mean of the following expressions
\begin{eqnarray*}
	c\approx \frac{\widehat{\v Y}}{(\widehat{\E Y})^\beta}, \quad\text{with}\quad
	\widehat{\E Y}=\overline Y\ =\ \frac 1k\sum_{i=1}^kY_i,\quad\text{and}\quad
	\widehat{\v Y}= \frac 1{k-1}\sum_{i=1}^k(Y_i-\overline Y)^2.
\end{eqnarray*}

Here, the expressions of the variance $\v Y$ and $\E Y $ make sense since those parameters provide a first approximation to the distributions of i.i.d. samples. Another analogous expression emerges in actuarial and financial sciences related to the so-called measures of variation, e.g. \cite{albrecher2007asymptotic, albrecher2010asymptotics}. It is strongly related to the ratio
\begin{eqnarray*}
	d\approx \frac{\widehat{\E Y^2}}{(\widehat{\E Y})^\beta}, \quad \text{with}\quad
	\widehat{\E Y^2}=\frac 1k\sum_{i=1}^kY_i^2.
\end{eqnarray*}

In \cite{albrecher2010asymptotics}, the convergence of such self-normalised sums is investigated in details.  For instance, it is  proved that  the convergence in distribution of suitably normalised ratios holds for heavy or light tailed distributions.  In  \cite{albrecher2007asymptotic} the above approximation is proved to  holds in $\L^p$ through the limit expression for each of the moments of the above ratio.  Unlike these examples, our paper is principally concerned with dependent random variables that have received less attention in the literature. Indeed, the theoretical literature has thus far, treated the independent case, also referred to as the classic Taylor's law. However, we will restrict ou work to random variables with moments with order greater that 4, but we will deal with dependent random variables and thus handle the case of time series. 
To the best of our knowledge, the current paper is the first attempt to include the dependence structure in such Taylor's laws. To achieve this, we first consider stationary and ergodic processes $(X_i)_{i\in\Z}$ to decompose the sequence of interest into Bernstein blocks. This allows to divide the data up into blocks in such a way to control adequately the dependence between blocks. 
This will be crucial to investigate the asymptotic distribution of the considered quantities in order to derive statistical properties for the two different Taylor's laws (static and dynamic). Formally, we show that the above expressions admit convergent behaviours. This follows the same idea as for the ``classic'' behaviour of such laws. Indeed, the variance involved in the above mentioned expressions has a counterpart in the weakly dependent cases. Namely, as stressed out for example in \cite{D1994,DL1999,Rio} or \cite{D2018}, under an adapted weak dependence assumption the partial centred sums renormalised with a $\sqrt n-$factor are proved to converge to a centred Gaussian distribution with variance $\sigma^2$ such that
\begin{eqnarray}\label{sigma}
\sigma^2&=& \sum_{j=-\infty}^\infty \cov(X_0,X_j).
\end{eqnarray} 
Accordingly, the extension of the Taylor's law will inevitably incorporate the series of covariances and not just over the marginal distribution. By doing so, we advocate that such an index has a different meaning that the usual Taylor's law which only depends on marginal distributions.  However, this topic exceeds the scope of this paper. Formally, will be concerned with limit behaviours of such a new index. The results developed throughout the paper take into account this new dynamic exponent as well as the classical (static) Taylor's exponent; both are considered for general classes of dependent random processes.  

The paper is organised as follows. First, in Section \ref{s2} we introduce the empirical expressions necessary to deal both with the dynamic and static Taylor's laws in a dependent setting. In Section \ref{s_limit} we deal with limit theory under these laws. To this aim we describe a Bernstein's blocks technique used throughout the paper in order to control the dependence. Hence, the limit distribution of suitably normalised expressions for both static and dynamic indices is thus proved in Section \ref{s_limit}. Finally, Section \ref{statap} is dedicated to statistical applications which are respectively a test of goodness-of-fit  for the dynamic Taylor's law to hold, and a consistent estimation of those Taylor's exponents. This means that both results together will ensure a test for both  Taylor's laws to hold. The necessary dependence tools and technical results are introduced in the Appendices. 

	\section{Self-normalised sums}\label{s2}
Let $(Y_i)_{i\in\Z}$ and $(X_i)_{i\in\Z}$ be two sequences of non-negative and identically distributed random variables. Since the paper is aimed at looking at dependent and dynamic samples, we will denote, henceforth, by $(Y_i)_{i\in\Z}$ the statistics under consideration and the sequence $(X_i)_{i\in\Z}$ will be associated to the classic Taylor's law. Here, we first recall the statistics associated with the usual Taylor's laws. Formally, with $k>0$, we define $S_k$ as the ratio
\begin{eqnarray*}
	S_k&=&\frac{\sum_{j=1}^kY_j^2/k}{\Big(\sum_{j=1}^kY_j/k\Big)^2}=k\frac{\sum_{j=1}^kY_j^2}{\Big(\sum_{j=1}^kY_j\Big)^2}.
\end{eqnarray*}
Hence, with $\overline Y=(\sum_{j=1}^kY_j)/k$ we can write
\begin{equation}\label{sk}
S_k=\frac{\sum_{j=1}^kY_j^2/k}{\overline Y^2}=\frac{k-1}k\cdot  T_k+1,
\end{equation}
where we denote by $T_k$ the Taylor's law statistics defined as
\begin{eqnarray}\label{tk}
T_k&=&\frac{\sum_{j=1}^k(Y_j-\overline Y)^2/({k-1})}{\overline Y^2}.
\end{eqnarray}
This is a plug-in estimate of 
$T={\sigma^2}/{m^2}$, with $m=\E Y_1$ and $\sigma^2=\v Y_1$.
Notice that since we can write
\begin{eqnarray}\label{expre}
T_k&=&\frac{k}{k-1}\cdot (S_k-1),\end{eqnarray}
the asymptotic behaviour of $T_k$ results may be plugged into those for $S_k$. 
{The above relation provides us with a link between results for the self-normalised statistics $S_k$ and for Taylor's statistics $T_k$.}
More generally for $\beta >0$, as in \cite{cohen2012taylor}, we consider Taylor's law with order $\beta$ in case $\sigma^2=cm^\beta$, whose corresponding statistics can be written as
\begin{eqnarray*}
	T_{k,\beta}&=&\frac{\sum_{j=1}^k(Y_j-\overline Y)^2/({k-1})}{\overline Y^\beta}.
\end{eqnarray*}
In this case, set $W_k=\sum_{j=1}^k(Y_j-\overline Y)^2/({k-1})$, then:
\begin{eqnarray}\nonumber
S_{k,\beta}=\frac{\sum_{j=1}^kY_j^2/k}{\overline Y^\beta}&=&\frac{k-1}k\cdot  T_{k,\beta}+\overline Y^{2-\beta}
=\frac{k-1}k\cdot \frac{ \sum_{j=1}^k(Y_j-\overline Y)^2/({k-1})}{\overline Y^{\beta}}+\overline Y^{2-\beta}\\
\nonumber
&=&\left(\frac{k-1}k\cdot  T_{k}+1\right)\overline Y^{2-\beta}	
\end{eqnarray}

\noindent
In the dependent framework of a stationary time series, we make use of the Bernstein's block idea to divide the sample $X_1,\ldots, X_n$, for $n>0$, into blocks of a given size in such a way to control the dependence between the blocks. To this end, we consider an integer $p_n\in\{1,\ldots,n\}$ and let $k_n=[n/p_n]$. We then denote $Y_i^{(n)}$ the sequence of partial sums over observations in block $B_{i,n}$, i.e.

\begin{eqnarray}\label{depen}
Y_i^{(n)}&=&\frac1{p_n}\sum_{j\in B_{i,n}}X_j, \quad
B_{i,n} = [(i-1)p_n+1,ip_n]\bigcap \N ,  \qquad
1\le i \le k_n. 
\end{eqnarray}
Hence, the statistics under consideration defined in \eqref{tk} and \eqref{sk} with $k=k_n$ can now be denoted respectively as $S^{(n)}_\beta$ and $  T_\beta^{(n)}$:
\begin{eqnarray*}
	{ \overline Y}^{(n)}&=&\frac1{k_np_n}\sum_{i=1}^{k_n}  \overline Y_i^{(n)},\\
	S_\beta^{(n)}&=&\frac{\sum_{j=1}^{k_n}(Y^{(n)}_j)^2/k_n}{(\overline Y^{(n)})^\beta},\\
	T_{\beta}^{(n)}&=&\frac{\sum_{j=1}^{k_n}(Y^{(n)}_j-\overline Y^{(n)})^2/({k_n-1})}{(\overline Y^{(n)})^\beta}.
\end{eqnarray*}
Moreover for the sake of homogeneity we denote $S^{(n)}$ the expression $S_2^{(n)}$.
\begin{rmk}[Bernstein's Blocks]  First, the reader is deferred to Appendix \ref{ss_second order} for a second order analysis of the behaviour of partial sums processes in \eqref{depen}, and \ref{ss_moment} for an higher order analysis.

		Second, we should note that, to the cost of an additional block with size less than $p_n$, we can also use all the data set, i.e. $X_1,\ldots,X_n $, by setting $ B_{k_n+1,n}=[k_np_n+1, n]$. By doing so, we do not affect the behaviour of partial sums when $k_n\to \infty$. Indeed, if the condition $a>1$ is fulfilled in \eqref{deccov}, see Appendix \ref{ss_second order}, we can show that
		$$
		\v \Big(\sum_{j\in B_{k_n+1,n}}X_j\Big)={\cal O}(p_n)\ll n,
		$$
		which entails that partial sums up to $n$ behave the same way as sums over all blocks with size $p_n$ as soon as $\sigma^2\ne0$. In fact, in this case, we shall have $\lim_{p\uparrow\infty}\v \big(\sum_{i=1}^pX_i/\sqrt{p}\big)= \sigma^2 $. Therefore, for the sake of readability, we will not consider this correction term in the sequel. 	
\end{rmk}
\begin{rmk}[Dynamic Taylor's law with exponent $\beta$]
	For $\beta >0$, we extend the above expressions for dynamic Taylor's law with order $\beta$. Recall that in this case, we are expecting a relationship of the form $\sigma^2=cm^\beta$. In fact, $ T^{(n)}$ is associated with Taylor's law with exponent $\beta=2$ and the following simple algebraic relation allows to consider all the possible exponents $\beta$. In fact, we shall remark that:
	\begin{equation}\label{skn}
	S_{\beta}^{(n)}=\frac{\sum_{j=1}^{k_n}\big(Y_j^{(n)}\big)^2/k_n}{\big(\overline Y^{(n)}\big)^\beta}= \left(\frac{k_n-1}{k_n}\cdot  T^{(n)}+1\right)(\overline Y^{(n)})^{2-\beta},
	\end{equation}
	which allows us to consider the Taylor's laws for more general settings.
\end{rmk}

\section{Limit theory in distribution}\label{s_limit}
Note that in this dependent framework $\overline Y\equiv \overline X^{(n)}=(X_1+\cdots+X_n)/n$ is simply the empirical mean of the observed process $(X_i)_{i\in \Z}$.
Under basic ergodic assumptions we have:
$$
\lim_n \overline X^{(n)}=\E X=m, \mbox{ a.s.}
$$
Hence, the asymptotic behaviour of the expression $T^{(n)}$ corresponding to $T_k$ in \eqref{expre} for this dependent setting  is that of 
\begin{eqnarray*}\label{expr}
\bar T^{(n)}&=&\frac1{m^2}\cdot
\frac1{k_n}\sum_{i=1}^{k_n}(Y_i^{(n)}-\overline X^{(n)})^2,
\end{eqnarray*}
which means that $\lim_n T^{(n)}/\bar T^{(n)}=1$. Henceforth, we let
\begin{eqnarray*}\label{exprX}
\tilde  T^{(n)}&=&\frac1{m^2}\cdot\frac1{k_n}\sum_{i=1}^{k_n}(Y_i^{(n)}-m)^2.
\end{eqnarray*}
Using assumption \eqref{sigma}, in case $\lim_np_n=\infty$, standard conditions imply that
\begin{eqnarray}\label{G_i}
G_{i,n}=\sqrt{p_n}(Y_i^{(n)}-m) , \qquad \qquad \mbox{for all }\ 1\le  i\le k_n,
\end{eqnarray} 
has a ${\cal N}(0,\sigma^2)-$standard Gaussian  asymptotic behaviour for each $ i\ge 1$.  
\\
Moreover, we have
\begin{equation}\label{gausslim}
\lim_n\E (G_{i,n})^2=\sigma^2, \qquad \qquad \forall  i\ge 1.
\end{equation}
\begin{rmk}
	Remark that the classic (or static) Taylor's law corresponds to $p_n=1$; in this case the corresponding block are no more asymptotically Gaussian and thus the above asymptotically Gaussian behaviour does not hold. Thus a separate discussion will be needed.
\end{rmk}

The relation \eqref{gausslim} entails that we need to center by force the expression of $G_{i,n}^2$. Thus, 
with the notation \eqref{G_i} we define the centred sequence
\begin{eqnarray}\label{U_i}
U_{i,n}&=& G_{i,n}^2-\E  G_{i,n}^2 , \qquad \qquad \mbox{for all }\ 1\le  i\le k_n,
\end{eqnarray}
thus,
\begin{eqnarray*}\nonumber
\tilde  T^{(n)}&=&\frac1{nm^2}\sum_{i=1}^{k_n}U_{i,n}+
\frac1{nm^2}\sum_{i=1}^{k_n}\E G_{i,n}^2,\\
\label{exprG}
&=&\frac1{nm^2}\sum_{i=1}^{k_n}U_{i,n}+
\frac1{p_nm^2}\sigma^2+\frac{E G_{i,n}^2-\sigma^2}{p_nm^2}.
\end{eqnarray*}
We now use  the  bound  \eqref{borErr}  to derive $\E G_{i,n}^2-\sigma^2={\cal O}\big(1/{p_n}\big)$. To this end, by setting
\begin{equation}\label{G^n}
{\mathcal G}_{n}=
\frac1{m^2\sqrt{k_n}}\sum_{i=1}^{k_n} U_{i,n},
\end{equation}
we deduce that \begin{eqnarray*}
	g^{(n)}&\equiv& \sqrt{k_n}\left(p_n\tilde  T^{(n)}- \frac{\sigma^2}{m^2}\right)
	\ =\
	{\mathcal G}_{n}+
	{\cal O}\Big(\frac {\sqrt{k_n}}{p_n}\Big),
\end{eqnarray*}
or equivalently,
\begin{eqnarray*}
g^{(n)}&=& \label{decom}
{\mathcal G}_{n}+
{\cal O}\Big(\sqrt{\frac {n}{p_n^3}}\Big),
\end{eqnarray*}
\noindent since $n=k_np_n$. 

Next, we will prove in Theorem \ref{lemmeG}  that ${\mathcal G}_{n}$ admits a Gaussian asymptotic behaviour ${\cal N}(0,\Sigma^2)$  by using Lemma 3 in \cite{BDLR}.  We will refer to the Appendix \ref{ss2} to derive the necessary dependence conditions.
To this end, assume that for some $r>4$, $\E|X_0|^r<\infty$; then Lemma \ref{mom4} (see also \cite[Equation (4.2.6)]{DDLLLP}) implies that $|\cov (U_{0,n},U_{q,n})|\le C(\theta_n^U(q))^{\frac{r-2}{r-1}} $ from weak dependence conditions for $q\ne0$.
Moreover conditions for moments of  $G_{i,n}$  with order $\delta >2  $ to be  bounded are given in  Lemma \ref{mom4}.
Now in order to derive Lemma \ref{expressionSigma} we note that
Lemmas \ref{cov2} and \ref{cov3} provide us with conditions to ensure the existence of some $r>4$ such that $\|X_0\|_r<\infty $. In fact, one needs conditions \eqref{momentalpha}, or \eqref{momenttheta} to hold, as well as the following limit behaviours
\begin{eqnarray*}
	\lim_{n\to \infty} p_n^2\sum_{\ell=1}^\infty \alpha^{\frac{r-4}{r}}(p_n\ell )=0,
	\qquad\mbox{and}\qquad
	\lim_{n\to \infty} p_n^2\sum_{\ell=1}^\infty \theta^{\frac{r-4}{r-2}}(p_n\ell )=0.
\end{eqnarray*}
Under such conditions we can state the following lemma.
\begin{lemma}\label{expressionSigma}
	Assume that  $\|X_0\|_r<\infty $ for some $r>4$, and either  $\alpha(q)={\cal O}(q^{-\alpha})$ or $\theta(q)={\cal O}(q^{-\theta})$ holds for
	$\displaystyle \alpha>2\cdot \frac{r}{r-4}$, or  $\displaystyle \theta>2\cdot \frac{r-2}{r-4}$.  Then we have
	\begin{eqnarray*}
	\lim_{n\to\infty}\sum_{\ell\ne0} |\cov(U_{0,n},U_{\ell,n})|=0.
	\end{eqnarray*}
	Also we have
	\begin{eqnarray*}\nonumber
		\Sigma^2
		&=&\frac1{m^4}\lim_n\E( U_{0,n})^2.
	\end{eqnarray*}
\end{lemma}
\begin{proof}
	Letting $\widetilde X_i=X_i-m$ for $i\ge1$, we can write:
	\begin{eqnarray*}\nonumber
		\Sigma^2&=&\frac1{m^4}\lim_n \frac 1{k_n}\v \left(\sum_{i=1}^{k_n} U_{i,n}\right)
		=\frac1{m^4}\lim_n\E( U_{0,n})^2
		=\frac1{m^4}\lim_n\v G^2_{0,n},
		\\
		&=&\frac1{m^4}\lim_n\frac 1{p_n^2}\sum_{i,j,i',j'=1}^{p_n}\cov ( \widetilde X_i\widetilde X_{i'},  \widetilde X_{j}\widetilde X_{j'})
		=\frac1{m^4}\sum_{i,j'=1}^{p_n}\cov (\widetilde X_0 \widetilde X_i,  \widetilde X_0\widetilde X_{j}).
	\end{eqnarray*}

	\begin{rmk}[Cumulants]\label{cumdef}
		Notes that this last expression is related to the cumulants	 
		$\kappa(X_0,X_u,X_v,X_w)$. Recall that $\kappa(X,Y,Z,T)$ is the coefficient of $t_1t_2t_3t_4$ in the Taylor expansion of $\log \E  \exp(i t\cdot {\mathbf V}) $ if $t=(t_1,t_2,t_3,t_4)$ and ${\mathbf V}=(X,Y,Z,T)$. Note that if the process is Gaussian then the cumulants with order great that 2 all vanish.  In any case using \cite{LS59} (see also \cite{Ros2000} or \cite{BDL}) we can show  hat if all the moments are well defined, then 
		$$\cov (  XY,  ZT)=\kappa(X,Y,Z,T)+\cov(  X,Z)\cov(  Y,T)+\cov(  X,T)\cov( X,Z),$$
		since 
		\begin{eqnarray*}
			\kappa(X,Y,Z,T)&=&\cov(  XY, ZT)-\cov(  X,Y)\cov(  Z,T)\\ &-&\cov( X,Z)\cov(  Y,T)-\cov(  X,T)\cov( X,Z).
	\end{eqnarray*}
\end{rmk}

	Thus Remark \ref{cumdef} implies that
	\begin{eqnarray*}
		\cov (  X_i\!X_{i'},  X_{j}\!X_{j'})\!&=&\! \kappa(\!X_i,\!X_{i'}, \! X_{j},\!X_{j'}\!)\!+\!\cov (  X_i, X_{j})\cov( X_{i'}X_{j'})\!\\&+&\!\cov (  X_i ,X_{j'})\cov( X_{i'},X_{j}).
	\end{eqnarray*}
	It thus follows that we may write
	$m^4\Sigma^2=\lim_nA_n$ with
	\begin{eqnarray*}
		A_n&=&\frac1{p_n^2}  \sum_{i,j,i',j'=1}^{p_n}\cov (  X_iX_{i'},  X_{j}X_{j'})=2 B_n+C_n,\\
		B_n&=& \frac1{p_n^2}  \sum_{i,j,i',j'=1}^{p_n}\cov (  X_i, X_{j})\cov( X_{i'}X_{j'}),\\
		C_n&=& \frac1{p_n^2}  \sum_{i,j,i',j'=1}^{p_n}\kappa (  X_i, X_{j},X_{i'}X_{j'}).\\
	\end{eqnarray*}
	Hence it is easy to prove that  $ \lim_n B_n =\sigma^4$, and $ \lim_n C_n =0$ when the cumulant sums condition \eqref{cumulants} holds. This assumption writes:
	\begin{eqnarray}\label{cumulants}
	\sum_{i,=1}^{\infty} \sum_{j=1}^{\infty} \sum_{k=1}^{\infty}\ |\kappa (i,j,k)|<\infty, \qquad\mbox{with } \kappa (i,j,k)=\kappa (  X_0, X_{i},X_{j}X_{k}).
	\end{eqnarray}
	We have thus proved, using Lemma \ref{expressionSigma}, that if  \eqref{cumulants} holds then Proposition \ref{ESigma} allows to specify the limit variance $\Sigma$ in the CLT.
\end{proof} 
\begin{rmk}[Sufficient conditions]
	The condition \eqref{cumulants} is widely discussed in \cite{Ros1985} and Theorem 4 on p.  138 provides a sufficient condition for \eqref{cumulants} to hold, see also \cite{Ros2000}.  This condition is also used as condition {\bf M} in \cite{BDL}. In this work a precise study provides the reader with sufficient strong mixing conditions and under $\theta-$weak dependence. More precisely,  if $\E|X_0|^r<\infty$, then the condition \eqref{cumulants} holds if one of the following additional conditions hold: 
	\begin{eqnarray*}
		\sum_{j=1}^\infty j^{\frac 2{r-4}}\alpha(j)<\infty,\qquad\mbox{and}\qquad
		\sum_{j=1}^\infty \theta^{\frac{r-4}{r-1}}(j)<\infty.
	\end{eqnarray*} 
	We should note that those conditions are not necessary for \eqref{cumulants} but they are implied by the assumptions in Theorem \ref{lemmeG}. 
\end{rmk}

\begin{prop}\label{ESigma}
	Assume that conditions in  Lemma \ref{expressionSigma} and \eqref{cumulants}  hold then:
	\begin{equation}\label{bigsigma}
	\Sigma^2=2\cdot\frac{\sigma^4}{m^4}.
	\end{equation}
\end{prop}
Now assume that $\lim_{n\to\infty} {n}/{p_n^3}=0$, simply let 
\begin{equation} \label{zeta} 
p_n=[\kappa n^{\zeta}],\ \mbox{ with }\ \zeta>\frac13. 
\end{equation}
Then we will prove that:
\begin{eqnarray}\label{interm}
\sqrt{k_n}\bar T^{(n)}={\cal G}_n+
\frac{\sigma^2}{m^2}+
\sqrt{k_n}
(\bar T^{(n)}-\tilde T^{(n)})+
{\cal O}\Big(\sqrt{\frac {n}{p_n^3}}\Big).
\end{eqnarray}
For this we remark  that the centred Taylor's statistics may be decomposed as follows \begin{eqnarray*}\label{i2}
\bar T^{(n)}-\tilde T^{(n)}=\frac1{m^2}(\overline X^{(n)}-m)\left(
\frac1{k_n}\sum_{i=1}^{k_n}(Y_i^{(n)}-m)
+(\overline X^{(n)}-m)\right).
\end{eqnarray*}
The bound $\overline X^{(n)}-m= {\cal O}(1/\sqrt{n})$ holds from \eqref{clt}. 
Similarly, we have 
$\v (Y_i^{(n)}-m)= {\cal O}(1/{p_n})$
and if \eqref{deccov}  holds for $a>2$ then $\displaystyle\v \Big(\sum_{i=1}^{k_n}(Y_i^{(n)}-m)\Big)={\cal O}(k_n/{p_n}).$ Indeed, we can see that
\begin{eqnarray}\nonumber
\v \left(\sum_{i=1}^{k_n}Y_i^{(n)} \right)&=&s\sum_{|\ell|<k_n}(k_n-|\ell|)\cov(Y_0^{(n)},Y_\ell^{(n)}),\\
&\le&{\cal O}\left(\frac{k_n}{p_n}\right)
+2k_n\sum_{\ell=1}^{k_n}|\cov(Y_0^{(n)},Y_\ell^{(n)})|\label{eq2}
\end{eqnarray} 
Thus, \begin{eqnarray}
\label{in1}
\v \left(\sum_{i=1}^{k_n}Y_i^{(n)}\right)={\cal O}\left(\frac{k_n}{p_n}\right).
\end{eqnarray}
This is, however, not straightforward as it needs a thorough investigation of the second term in left-hand side of \eqref{eq2}. In fact, in order to prove this, we need to decompose $Y_0^{(n)}=Y_-+Y_+$ such that $$Y_-=\frac1{p_n}\sum_{i=1}^{p_n-q_n}X_i, \qquad Y_+=\frac1{p_n}\sum_{i=p_n-q_n+1}^{p_n}X_i,
$$ for some $q\equiv q_n<p_n$, which will be specified  on the sequel. Then with $\tilde k_n=[n/(p_n+q_n)]$ we obtain
\begin{eqnarray*}
	|\cov(Y_0^{(n)},Y_\ell^{(n)})|\le|\cov(Y_-,Y_\ell^{(n)})|+|\cov(Y_+,Y_\ell^{(n)})| .
\end{eqnarray*}
Indeed,  if \eqref{deccov}  holds for  $a>1$ then,
$ \v Y_\ell^{(n)}={\cal O}(1/{p_n})$ and thus leads to
\begin{eqnarray*}
	|\cov(Y_+,Y_\ell^{(n)})| &\le&\frac1{p_n}\sum_{j=p_n-q_n+1}^{p_n}|\cov(X_j,Y_\ell^{(n)}| 
	\le  \frac {q_n}{p_n} \sqrt{\v X_0\cdot \v Y_\ell^{(n)}}.
\end{eqnarray*}
This writes 
\begin{eqnarray}\label{16a}
	|\cov(Y_+,Y_\ell^{(n)})| 	&=&{\cal O}\left(\frac{q_n}{p_n\sqrt{p_n}}\right).
\end{eqnarray}Hence, the sum of the $k_n$ corresponding terms admits a contribution with order  $  q_n\tilde  k_n/p_n\sqrt{p_n}$. In case  $q_n=  {\cal O}(\sqrt{p_n}/\tilde k_n)= {\cal O}(p_n^{3/2}/n)$ this contribution admits the order $ {\cal O}\left(1/p_n\right)$. Moreover, letting $q_n=[n^\nu]$, the previous inequality holds in case $0<\nu\le 3\zeta-1$, which is only possible when $\zeta>\frac13$ in \eqref{zeta}. Now using the fact that $q_n<p_n$ we have
\begin{eqnarray*}
	|\cov(Y_-,Y_\ell^{(n)})| &\le&\frac1{p_n^2}\sum_{j=p_n-q_n+1}^{p_n}\sum_{u\in B_{\ell,n}}|\cov(X_j,X_u)| ,
	\\
	&=&{\cal O}\Big(\frac {q_n}{p_n^2}\Big)\sum_{i\ge \ell p_n} i^{-a}
	={\cal O}\left(\frac {q_n}{p_n^2}(\ell p_n)^{1-a}\right)=\ell^{1-a}{\cal O}\left(\frac {q_n}{p_n} p_n^{-a}\right),
\end{eqnarray*}
thus,
\begin{eqnarray}\label{16b}
	|\cov(Y_-,Y_\ell^{(n)})| 
	&=&\ell^{1-a}{\cal O}\left(p_n^{-1}\right),
\end{eqnarray}

since $a>1$. Hence, from summation, the relations \eqref{in1} with $a>2$ and \eqref{16b} together imply 
\begin{eqnarray*}
	\sum_{\ell=1}^{k_n}|\cov(Y_-,Y_\ell^{(n)})| &=&{\cal O}\left(p_n^{-1}\right).
\end{eqnarray*}
Finally, this allows to conclude that the relation  \eqref{in1} holds for some $a>2$. Accordingly, the  relations  \eqref{bor2} and  \eqref{in1} together imply 
\begin{eqnarray}\label{i1}\bar  T^{(n)}-\tilde T^{(n)}= {\cal O}\Big(\frac1{n}+\frac{\sqrt k_n}{k_n\sqrt {p_nn} }\Big)= {\cal O}\Big(\frac1{n}\Big). \end{eqnarray}
Now, if we go back to \eqref{interm}, we can show that if we assume that $a>2$ in \eqref{deccov} we can write 
\begin{eqnarray}\label{bound}
\sqrt{k_n}\bar T^{(n)}=G^{(n)}+
\frac{\sigma^2}{m^2}+
\sqrt{k_n}
(\bar T^{(n)}-\tilde T^{(n)})+
{\cal O}\Big(\sqrt{\frac {1}{k_n}}+\sqrt{\frac1 {n}}\Big),
\end{eqnarray}
thus $G^{(n)}$ converges to $ {\cal N}\left(0, \Sigma^2\right)$  with $\Sigma^2$ defined from \eqref{bigsigma} provided that $\frac13<\zeta<1$ in \eqref{zeta}.

\begin{thm}\label{lemmeG}
	With the notations \eqref{U_i} and \eqref{G^n}, assume  that for some $r>4$, $\|X_0\|_r<\infty $ and that \eqref{cumulants} holds as well as 
	$\lim_n (k_n^3q_n/n)=0$. If moreover one of the following conditions is fulfilled:
	$$\begin{array}{lllllll}
	\alpha(q)&=&{\cal O}(q^{-\alpha}),& \text{ with }  &\displaystyle \alpha>2\cdot \dfrac{r}{r-4}, &  \text{and }&  \lim_n k_n\alpha(q_n)=0\;,
	\\
	\theta(q)&=&{\cal O}(q^{-\theta}),& \text{ with } &  \displaystyle \theta>2\cdot \frac{r-2}{r-4},  &  \text{and }& \lim_n nk_n\theta^{\frac{r-2}{r+2}}(q_n)=0,
	\end{array}
	$$
	then $${\mathcal G}_{n}\underset{n\to \infty} {\to} {\mathcal N}\Big(0,2\cdot\frac{\sigma^4}{m^4}\Big), \qquad \mbox{ in distribution.}
	$$
\end{thm}
\begin{proof} 
	According to notations \eqref{U_i} and \eqref{G^n} we set, for $1\le j\le k_n$,
	\begin{equation*}\label{G^n_j}
	{\mathcal G}_{j,n}=\frac1{m^2\sqrt{k_n}}\sum_{i=1}^{j} U_{i,n}.
	\end{equation*}
	\\
	A dependent version of Lindeberg lemma (see Lemma 3 in \cite{BDLR}) 
	requires the existence of some $\gamma>2$ and that  the three   following conditions hold	
	\begin{eqnarray}
	\label{lind3}
	\lim_{n\to\infty}\frac1{m^4}\lim_n \frac 1{k_n}\v \left(\sum_{i=1}^{k_n} U_{i,n}\right)&=&\Sigma^2\qquad \mbox{ exists,}
	\\
	\label{lind1}
	\lim_{n\to\infty}\sum_{j=2}^{k_n}|\cov(
	e^{it{\mathcal G}_{j,n}}, 
	e^{itU_{j,n}})|&=& 0,
	\\
	\label{lind2}
	\lim_{n\to\infty}k_n^{-\frac\gamma2}\sum_{j=1}^{k_n}\E\left|U_{j,n}\right|^\gamma
	&=& 0.
	\end{eqnarray}
	We will thus successively consider each of those three relations.
	\paragraph*{\underline{Relation \eqref{lind3}}} The relation \eqref{lind3} is proved in Proposition \ref{ESigma} together with the expression  $$\Sigma^2=2\cdot\frac{\sigma^4}{m^4}.$$
	\paragraph*{\underline{Relation \eqref{lind1}}} The term \eqref{lind1} is somehow tricky. First notice that
	\begin{eqnarray*}
		\cov(e^{it {\mathcal G}_{j,n} }, e^{itU_{j,n}})&=&
		\cov( e^{it{\mathcal G}_{j,n}}-e^{it{\mathcal G}_{j-1,n}}, e^{itU_{j,n}})
		+\cov(e^{it{\mathcal G}_{j-1,n}}, e^{itU_{j,n}}), 
	\end{eqnarray*}
	and since
	\begin{eqnarray*}
		\E|e^{it{\mathcal G}_{j,n}}-e^{it{\mathcal G}_{j-1,n}}|
		&\le &\frac1{m^2\sqrt{k_n}}\E|U_{j,n}|={\mathcal O}( \frac1{\sqrt{k_n}}),
	\end{eqnarray*} 
	we obtain
	\begin{eqnarray*}
		|\cov( e^{it{\mathcal G}_{j,n}}-e^{it{\mathcal G}_{j-1,n}}, e^{itU_{j,n}})|&\le&  \frac2{m^2\sqrt{k_n}}\E|U_{j,n}|={\mathcal O}( \frac1{\sqrt{k_n}}).
\end{eqnarray*}
	Then summing up $k_n$ terms as above   provides an expression with order ${\cal O}(\sqrt{k_n})$ which does not  tend to 0. This means that additional work has to be processed to derive \eqref{lind1}. Consider thus the following decomposition
	${\mathcal G}_{j,n}={\mathcal G}_{j,n}+({\mathcal G}_{j-1,n}-A)+{\mathcal G}_{j-1,n}+A$, with $A=G^2-\E G^2$ and $G=\frac1{\sqrt{p_n}}\sum_{i=jp_n-q_n}^{jp_n-1}X_i$, such as the term $G$ is negligible.  Remark that ${\mathcal G}^- ={\mathcal G}_{j-1,n}+A$ is $q_n$-distant from 
	${\mathcal G}_{j,n}$, and
	$ {\mathcal G}_{j,n}-({\mathcal G}_{j-1,n}+A)=G_{j,n}^2-G^2-\E (G_{j,n}^2-G^2)$. Therefore, we have 
	$$\E|G_{j,n}^2-G^2|\le \E|G_{j,n}-G|^2+2\E|(G_{j,n}-G)G_{j,n}|+2\E|(G_{j,n}-G)G|={\mathcal O}(\sqrt{\frac{q_n} {p_n}}). $$
In order to prove  \eqref{lind1} we first need $\lim_nk_n\sqrt{\frac{q_n} {k_np_n}}=0$  which holds if  $\lim_n{q_nk_n} /{p_n}=0$. 
	This is achieved when $p_n\sim n^u$ and  $q_n\sim n^v$ provided that $u>\frac12$ and  $0<v<2u-1$.
	Finally, what is left is to bound the second term, namely
	\begin{eqnarray*}
		|\cov(e^{it{\mathcal G}_{j-1,n}},  e^{itU_{j,n}})|&\le&|\cov(e^{it{\mathcal G}^-},  e^{itU_{j,n}})|+|\cov(e^{it{\mathcal G}_{j-2}^{(n)}}(e^{itU_{j-1,n}}-e^{itA}),  e^{itU_{j,n}})|,\\
		&\le&|\cov(e^{it{\mathcal G}^-},  e^{itU_{j,n}})|+2t\E|  G_{j,n}^2-G^2-\E (G_{j,n}^2-G^2) |,
		\\
		&\le&|\cov(e^{it{\mathcal G}^-},  e^{itU_{j,n} })|+{\mathcal O}(\sqrt{\frac {q_n} {p_n}}). \end{eqnarray*}
	To this end, we distinguish the two following cases.\\
	 
		\noindent\rmi In the strong mixing case
		\begin{eqnarray*} |\cov(e^{it{\mathcal G}^-},\  e^{itU_{j,n}})|\le   \alpha(q_n). \end{eqnarray*}
		Thus condition \eqref{lind1} occurs in case both conditions $\lim_n k^2_nq_n/p_n=0$ and  $\lim_n k_n\alpha(q_n)=0$ are fulfilled.\\
		
		\noindent\rmii In the $\theta-$weakly dependent case the situation is more intricate since the  heredity of such conditions is less straightforward. Set  $e^{itU_{j,n}}=f\circ g\circ h((X_i)_{i\in B_{j,n}})$, with $f(z)=e^{itz}$, $g(z)=z^2-\E G_{1,n}^2$ and $h(x_1,\ldots, x_{p_n})=\frac1{\sqrt{p_n}}\sum_{i=1}^{p_n}x_i$. Hence $\lip f=|t|$ and $\lip h=\frac1{\sqrt{p_n}}$. We let $\overline U_{j,n }$  $\overline G_{j,n }$ be the recentred truncations\footnote{
				Let $X$ be a real valued random variable and $M>0$.
				Set $\widehat X=\big( \widetilde X\vee(-M)\big)\wedge M$ and
				$\underline X=\widetilde  X-\widehat X$ at a level $M>0$; then 
				$|\underline  X|\le 2|\widetilde X|\1{(| \widetilde X|\ge M)}$. 
				A centred version of this truncation writes $\overline X=\widehat X -\E\widehat X$.} at a level $M>0$ to be precisely settled later of $U_{j,n }$  and $\overline G_{j,n }$, respectively.

		Then  with the help of Lemma \ref{mom4} we can write
		\begin{eqnarray*}
			|\cov(e^{it{\mathcal G}^- },  e^{itU_{j,n}})| &\le&
			|\cov(e^{it{\mathcal G}^- },  e^{it\overline U_{j,n}})
			+   |\cov(e^{it{\mathcal G}^- },  e^{itU_{j,n}}-e^{it\overline U_{j,n}})|,
			\\
			&\le& 2p_n\frac{|t| M^2} {\sqrt{p_n}} \theta(q_n)+ 2|t|\E|U_{j,n}-\overline U_{j,n}|,
			\\
			&\le& 2\sqrt{p_n}|t| M^2\theta(q_n)+ 2|t|\sqrt{\E |G_{j,n}+\overline G_{j,n}||G_{j,n}-\overline G_{j,n}|} ,
			\\
			&\le& 2\sqrt{p_n}|t| M^2\theta(q_n)+
			4|t|\sqrt{\E G^2_{j,n}E|G_{j,n}-\overline G_{j,n}|^2},\\
			&\le& 2\sqrt{p_n}|t| M^2\theta(q_n)+{\mathcal O}(1)
			\sqrt{p_n\E |X_0|^rM^{2-r}}\theta (q_n),
			\\
			&=&  {\mathcal O}( \sqrt{p_n}\theta^{\frac12\frac{r-2}{r+2}}(q_n)),\qquad\mbox{ with }\quad M=\theta^{-\frac2{2+r}}(q_n).
		\end{eqnarray*}     
		Thus condition \eqref{lind1} follows in case $\lim_n k^2_nq_n/p_n=0$ and  $\lim_n nk_n\theta^{\frac{r-2}{r+2}}(q_n)=0$.
		
	\paragraph*{\underline{Relation \eqref{lind2}}} Lemma \ref{mom4} allows to deal with condition \eqref{lind2} since 
	$$\E\left|U_{j,n}\right|^\gamma=\E|G_{j,n}^2-\E G_{j,n}^2|^\gamma\le 2^\gamma\E|G_{j,n}|^{2\gamma}={\mathcal O}(1),$$  follows from a convexity argument. Now \eqref{lind2} holds for $\gamma>2$ if a moment with order $\delta=2\gamma>4$ fits Lemma \ref{mom4}. This requires  that decays in the strong mixing or $\theta$-weak dependence both satisfy $\alpha,\theta\ge2\cdot\frac{r-2}{r-4}$ which  follows from the assumptions yielding \eqref{lind3}.\end{proof}

First notice that by using the weak Law of Large Numbers (LLN),  it is simple to prove that the asymptotic behaviours of  $\sqrt{k_n}p_nT^{(n)}$  and $\sqrt{k_n}p_n\bar T^{(n)}$ are analogous.
Then the central limit theorem $\sqrt{k_n}T^{(n)}\to {\cal N}\left(\frac{\sigma^2}{m^2}, 2\cdot\frac{\sigma^4}{m^{4}}\right)
$ follows from a direct use of Theorem \ref{lemmeG}, together with \eqref{skn} and \eqref{interm}. More generally, we can state the following result.
\begin{thm}\label{thP}
	Assume  that conditions in Theorem \ref{lemmeG}, then 
	we obtain
	\begin{equation*}\label{CLT}\displaystyle
	\sqrt{k_n}(T_\beta^{(n)}-m^{2-\beta})\underset{n\to\infty}{\rightarrow} {\cal N}\left(\frac{\sigma^2}{m^\beta}, 2\cdot\frac{\sigma^4}{m^{2\beta}}\right), 
	\end{equation*}
	in distribution.
\end{thm}
In order to check that the conditions in Theorem \ref{thP} are consistent  we state them for power decay cases.
\begin{cor}\label{corG}
	Let $p_n=[n^u]$, and $q_n=[n^v]$ for $0<v<u<1$. Assume that $u>2/3$ and $v<3u-2$.
	The conclusions of Theorem \ref{lemmeG}  hold  if for some $r>4$, $\|X_0\|_r<\infty $, \eqref{cumulants} holds as well as 
	I one of the following conditions:  
	$$\begin{array}{lllll}
	\alpha(q)&=&{\cal O}(q^{-\alpha}),& \text{ with }  &\displaystyle \alpha>2\cdot \dfrac{r}{r-4}\vee \frac{1-u}v,  
	\\
	\theta(q)&=&{\cal O}(q^{-\theta}),& \text{ with } &  \displaystyle \theta>2\cdot \frac{r-2}{r-4}\vee\Big( \dfrac{2-u}v\cdot \frac{r+2}{r-2}\Big).
	\end{array}
	$$
\end{cor}

\begin{rmk}[Static Taylor's law]\label{rmstatic}
	Notice that when $p_n=1$ we are left with the classic (or static) Taylor's law, for which $k_n=n$ and $Y_i^{(n)}=X_i$
	and $G_{i,n}=Y_i^{(n)}-m =X_i-m$.  
	In this case the above reasoning does not apply and some further modifications are needed in the proof for a CLT.
	Specifically \eqref{interm} does not hold a it requires $p_n>1$.  Thus, in order to handle this case we need to analyse throughly the corresponding expressions. In that case we develop
	$$
	{\mathcal G}_{n}=
	\frac1{m^2\sqrt{n}}\sum_{i=1}^{n} ((X_i-m)^2-\v X_0),
	\qquad
	\tilde  T^{(n)}
	={\mathcal G}_{n}+
	\frac1{m^2}\v X_0
	$$ 
	Thus using Appendix \ref{ss_moment} with $\delta=2$ entails that $
	\bar T^{(n)}-\tilde T^{(n)}=\frac2{m^2}(\overline X^{(n)}-m)^2={\cal O}(\frac1n)$. As a consequence we obtain
	$$\sqrt{n}\bar T^{(n)}={\cal G}_n+
	\frac{\sigma^2}{m^2}+
	{\cal O}\Big(\sqrt{\frac {1}{n}}\Big).
	$$
	Also, we need to recall that from   \eqref{i1}   $\lim_n\sqrt n\E|\bar  T^{(n)}-\tilde  T^{(n)}|=0.$ 
	Thus we have the following limit behavior: \begin{eqnarray*}
		\sqrt n(\tilde  T^{(n)}-\E\tilde  T^{(n)})&=&\frac1{m^2}\cdot\frac1{\sqrt n}\sum_{i=1}^{n}((X_i-m)^2-\v X_0)\longrightarrow{\cal N}(0,\Sigma_0^2/m^4), \end{eqnarray*}
	with $\Sigma_0^2=\sum_{j=-\infty}^\infty \cov ((X_0-m)^2,(X_j-m)^2),$	
	in case 
	$\E |X_0|^{r}<\infty$ for some $r>4$. Moreover, according respectively to \cite{DMR94}  or  to Corollary 1 in \cite{DD2003} and to the heredity result Proposition 2.1 in \cite{DDLLLP} (see also Appendix\, \ref{ss2}), 
	the CLT  holds for squared variables $(X_i-m)^2$. Hence, these make it possible to prove a central limit theorem when $p_n=1$. {In order to make the assumptions in the forthcoming result more simple, note that these conditions imply that Appendix\, \ref{ss_moment} holds for $\delta=2$.}
\end{rmk}
We can now state the central limit theorem for the case $p_n=1$.
\begin{thm}[Static Taylor's law]\label{thPS} Assume that $\E |X_0|^{r}<\infty$ for $r>4$ and  the following weak dependence or $\alpha$-mixing conditions hold
	\begin{eqnarray*}
		\sum_{j=1}^\infty j^{\frac 2{r-4}}\ \alpha(j)<\infty,\qquad\mbox{and}\qquad
		\sum_{j=1}^\infty j^{\frac 2{r-4}}\ \theta^{\frac{r-2}{r-1}}(j)<\infty.
	\end{eqnarray*} Then we have
	\begin{eqnarray*}\label{CLTstat}
	\sqrt{k_n}(T_\beta^{(n)}-m^{2-\beta})\ \underset{n\to\infty}{\to} \ {\cal N}\left(\frac{\sigma^2}{m^{\beta}}, \frac{\Sigma_0^2}{m^{2\beta}}\right), 
		\end{eqnarray*}
	in distribution, with
$$	\Sigma_0^2=\sum_{j=-\infty}^\infty \cov ((X_0-m)^2,(X_j-m)^2).$$

\end{thm}

\begin{rmk}[Practical sufficient conditions]
	If  respectively $\alpha(q)\le cq^{-\alpha}$, or $ \theta(q)\le cq^{-\theta}$ holds  then the conditions of Theorem \ref{thPS}  hold if 
	$$\alpha>\frac{r+2}{r-4} ,\mbox{ or }\theta>\frac{r+2}{r-4}\cdot \frac{r-1}{r-2} .$$
\end{rmk}
\begin{rmk}[Self-normalized expressions]
	The relation \eqref{skn} yields analogous results for the self-normalized  statistics $S_\beta^{(n)}$ in both  Theorems \ref{thP} and \ref{thPS}. \end{rmk}
\section{Statistical applications}\label{statap}
\subsection{Test of goodness-of-fit}\label{stattest}
For goodness-of-fit testing purposes, one needs to estimate empirically $m$, e.g. by the empirical mean $\widehat m=\overline X_n$; while many solutions are known in order to fit the limit  variance $\sigma^2$, by a convenient estimate $\widehat {\sigma^2}$ as it was suggested in \cite{DJL} and in the included references. By doing so, we have $\widehat m-m={\cal O}(1/\sqrt n)$ while the rate of convergence for $\widehat {\sigma^2}$ is nonparametric, since a typical estimate of $\sigma^2$ is the value at 0 of the spectral density of $X$ at the origin.
Now, multiplying the conclusion of Theorem \ref{thP}  by the constant $m^\beta/\sigma^2 $ we obtain 
\begin{eqnarray*}\label{CLTtest}
\sqrt{k_n}\cdot \frac{m^\beta}{\sigma^2} \left(T_\beta^{(n)}-m^{2-\beta} \right)\to {\cal N}\left(1, 2\right), \qquad \mbox{ in distribution.}
\end{eqnarray*}
This will be crucial in developing a testing procedure. Indeed, assume that $\tau$ is a quantile of the Normal distribution  such that $\P(|{\cal N}(0,2)|>\tau)=\eta$. A confidence interval is constructed as follows
\begin{eqnarray*}
	\lim\sup_{n\to\infty} \P( T_\beta^{(n)}\notin [a_n(\beta),b_n(\beta)])&\le& \eta,
\end{eqnarray*}
where
\begin{eqnarray*}
	a_n(\beta)= {\widehat m}^{2-\beta}+(1-\tau)\frac{1}{\sqrt{k_n}}\cdot \frac{{\widehat \sigma^2}}{{\widehat m}^\beta},
	\qquad
	b_n(\beta)= {\widehat m}^{2-\beta}+(1+\tau)\frac{1}{\sqrt{k_n}}\cdot \frac{{\widehat \sigma^2}}{{\widehat m}^\beta}.
\end{eqnarray*}
This leads to a test suitable to check whether $\beta=\beta_0$ or  $\beta\ne\beta_0$, by rejecting the hypothesis in case $ T_{\beta_0}^{(n)}\notin [a_n(\beta_0),b_n(\beta_0)]$, with an asymptotic level $\eta$. 
\begin{rmk}[Test for the static Taylor's law] From Theorem \ref{thPS}, we can derive analogous confidence intervals when $k_n=n$ corresponding the static Taylor's law. In fact, we have the following the asymptotic behavior
\begin{eqnarray*}
\sqrt{k_n}\cdot m^\beta\left(T_\beta^{(n)}-m^{2-\beta} \right)\to {\cal N}\left(\sigma^2, \Sigma_0^2\right), 
\end{eqnarray*}
in distribution. Hence a test of goodness-of-fit in this case may directly be considered. Similarly, this will also need the estimation of $\Sigma_0^2$.
\end{rmk}
\begin{rmk}[An heuristic for the case of contiguous alternatives]\label{remstat}
	This test may be proved to be asymptotically powerful under contiguous series of alternatives. Namely we may test hypotheses $\beta=\beta_0$ against $|\beta-\beta_0|\ge \delta_n$ for some $\delta_n\downarrow0$ with $\lim_n\sqrt k_n\delta_n/p_n=\infty$. We simply sketch its asymptotic power. To this end, assume that $\lim_n p_n^3/n= 0$.  Hence, we should prove that $\lim_n\sqrt {k_n}|T_\beta^{(n)}-T_{\beta_0}^{(n)}|=\infty$ as in the following heuristic  expressions:
	\begin{eqnarray} \nonumber T_\beta^{(n)}-T_{\beta_0}^{(n)}&=&
	\frac1{k_n}\sum_{j=1}^{k_n}(Y_j^{(n)})^2\left(\frac1{\big(\overline Y^{(n)}\big)^\beta}-\frac1{\big(\overline Y^{(n)}\big)^{\beta_0}}\right)
	\sim \E (Y_0^{(n)})^2\left(\frac1{m^\beta}-\frac1{m^{\beta_0}}\right)
	\\ \label{heuristic}
	&\sim& \frac{-\ln m\cdot m^{-\beta_0}}{p_n}(\beta-\beta_0)
	\end{eqnarray}
	Note that the function $f(\beta)=m^{-\beta}$ decays providing that $m>1$, which makes the above lower bound possible when $\beta>\beta_0+\delta_n$. Hence with a large probability there exists a constant $c$ such that: $$\sqrt{k_n}|T_\beta^{(n)}-T_{\beta_0}^{(n)}|\ge c \delta_n\frac{\sqrt{ k_n}}{p_n},$$ is unbounded.
	This, as well as different counter hypotheses, will be rigorously investigated in further works.
\end{rmk}
\begin{rmk}
	The chain or equivalences \eqref{heuristic} may also be useful to derive the estimation of $\beta_0 $ as well as a central limit theorem for $\frac{p_n}{\sqrt {k_n}}(\tilde \beta -\beta_0)$.
\end{rmk}

\noindent
It is worth pointing out that analogous results hold for the simpler static Taylor's law following the same lines.

\subsection{Estimation of Taylor's indices}\label{statest}

Estimation of $\beta$ is crucial in many applications. To this end, it is possible to derive from the above results the following.
\begin{cor} Assume that the assumptions of the Theorems \ref{thP} and \ref{thPS} hold, respectively. We also obtain that $T_\beta^{(n)}-m^{2-\beta}={\cal O}_{\P}(1/\sqrt {k_n})$. 
	Thus setting  $\widehat \beta_D=2-\ln T_\beta^{(n)}$ we obtain an estimator of $\beta $ for the dynamic Taylor's law such that $$\widehat \beta_D\to_{\P}\beta .\,$$ \end{cor}

\begin{rmk} All the chain of approximations to ${\cal G}_n$ is proved in $\L^2$. Moreover equation \eqref{lind3} proves that the sequence ${\cal G}_n$ is bounded in $\L^2$, which prove the first result. Note that the previous result also entails that $T_\beta^{(n)}/m^{2-\beta}\to 1$ in probability. Thus $$2-\ln T_\beta^{(n)}/\ln m\to \beta\,.$$
	Accordingly, we obtain an estimator of $\beta$, consistent in probability, in case $m$ is known. The convergence rate is again ${\cal O}_{\P}(1/\sqrt {k_n})$. \end{rmk}
	\begin{rmk} [Estimation of the Static Taylor's exponent]
	Thus in case $k_n=n$, using now Theorem \ref{thPS} and Remark \ref{remstat}, we set  $\widehat \beta_S=2-\ln T_\beta^{(n)}$ to obtain an estimator $\widehat \beta_S$   consistent  in probability of $\beta $ for the static Taylor's law.\\  The difference between both exponents relie of the fact that the statistic $T_\beta^{(n)}$ is currently considered  with $k_n=n$ while the Bernstein blocks needed in the dynamic Taylor's law rely on the relation $\lim_nk_n/n=0$.\end{rmk}
	\noindent
To the best of our knowledge, nothing exists so far in the literature related to the limit behaviour in distribution for such quantities. Thus testing hypotheses $\beta=\beta_0$ against $\beta\ne\beta_0$ is possible thanks to the tests described in Subsection \ref{stattest}.

\section*{Conclusions}
The present paper introduces  a new dynamic  Taylor's law. We prove a central limit theorem, in the dependent setting, for properly normalised relevant expressions related to both this new Dynamic Taylor's law as well as the static (or classic) one.  Our frame is however restricted to light tails processes and much more is needed to understand Taylor's laws under weaker assumptions. Many future issues will be considered in forthcoming publications. 
\begin{rmk}[Control of the moments]
	The study of  convergence in $\L^p$ is deferred to a forthcoming paper; let us simply quote that in the dependent cases it seems hard to reach low moment assumptions because of the systematic use of covariance inequalities. High moment assumptions allow to conduct equivalents for the moments of $T_\beta^{(n)}$ under dependence under heavy calculations.  Anyway this study is more adapted for developments in a more probabilistic review.
\end{rmk}	
\begin{rmk}[Comparing Taylor's laws]
	Empirical studies will also be conducted to investigate the respective domain of validity of those two different Taylor's law. We suspect that the new dynamic Taylor's law may be more relevant for some specific cases as those discussed in \cite{cohen2016population,saitoheffects}. Ecological considerations will be highlighted in such data studies. 
	
\end{rmk}

\section*{Acknowledgements}
We express our best thanks to Joel E. Cohen for exchanges and visits to Rockefeller University. P. Doukhan would like to thank Hansjoerg Albrecher for extremely fruitful discussions during  a visit of  in Lausanne, as well as ISFA and Columbia for various visits and exchanges. This work is funded by CY Initiative of Excellence (grant "Investissements d'Avenir" ANR-16-IDEX-0007), Project EcoDep. Y. Salhi's work is supported by the BNP Paribas Cardif Chair "Data Analytics \& Models
for Insurance" (DAMI). The views expressed herein are the author's owns and do not reflect those endorsed by BNP Paribas.

\bibliographystyle{imsart-number} 
\bibliography{Dynamic_Taylor_Law}       

\appendix
\section{Technical and Useful Tools}\label{appendix}
In this Appendix we set up the dependence considerations useful in the core of the paper.
\subsection{Notions of dependence}
\label{sec:Notion}

For the  Euclidean space $\R^d$ equipped with some norm $\|\cdot\|$ and for a function  $h:\R^d \to \R$, let us  denote by
$$
\mbox{Lip}(h)=\sup_{x\ne y} \frac{|h(x)-h(y)|}{\|x-y\|}.
$$
Define the space $\Lambda_1\big(\R^d\big)$  by the set of functions $h: \R^d \to \R$
such that $\mbox{Lip}(h)\le1$. Furthermore, let us denote by $\|h\|_\infty=\sup_{x\in \R^d}|h(x)|$. To be more specific, let $(\Omega, \mathcal{G}, \mbox{P})$ be a probability space. Following either \cite{D1994} or \cite{Rio}, recall  that for integers $1\le u,v\le \infty$,
the strong   mixing coefficient  is defined  by
\begin{eqnarray*}
\label{alphauv}\alpha_{u,v}(q)&=&\sup|\P(U\cap V)-\P(U)\P(V)| , \\
\label{alpha}\alpha(q)\quad&=&\alpha_{\infty,\infty}(q).
\end{eqnarray*}
Here, the supremum is taken over $U\in{\cal U}$, and $V\in{\cal V}$; 
with ${\cal U}=\sigma(X_{i_1},\ldots,X_{i_u})$, and ${\cal V}=\sigma(X_{j_1},\ldots,X_{j_v})$ for integers $i_1\le\cdots\le i_u\le i_u+q\le j_1\le\cdots\le j_v$; the suprema
first run over all such integers and second over the $\sigma$-fields ${\cal U}, {\cal V}$.
\\
On the other hand, the 
$\theta$ 
-coefficients (\cite {DD2003} or \cite{DDLLLP}) are  defined
as the least nonnegative number $\theta(q)$ such as 
\begin{eqnarray}\label{theta}
\big|\cov(f(X_{i_1},\ldots,X_{i_u}), g(X_{j_1},\ldots,X_{j_v}))\big|&\le& v\Lip g\,\|f\|_\infty \theta(q),
\end{eqnarray}
for integers  $i_1,\ldots, i_u,j_1,\ldots,j_v$  which satisfy $i_1\le\cdots\le i_u\le i_u+r\le j_1\le\cdots\le j_v$, and functions $f,g$ defined respectively
on the sets $(\R^d)^u$ and $(\R^d)^v$ equipped with the following norm
$$\|(x_1,\ldots,x_u)\|=\|x_1\|_\infty+\cdots+\|x_u\|_\infty, \qquad x_1,\ldots, x_u\in\R^d, $$
where $\|x\|_\infty=\max_{j}|x_{j}|$, for any $x \in \R^d$. The function $g$ is assumed to be Lipschitz.
\\
The sequence $(X_i)_i$ is said strong mixing or $\theta$-weakly dependent in case $\lim_q\alpha(q)=0$, or   respectively if $\lim_q\theta(q)=0$. In case of any doubt concerning the process under consideration those coefficients will respectively be denoted $\alpha_X(q)$ or  $\theta_X(q)$.

\begin{rmk}
	From a simple inclusion $\alpha_Y(q)\le \alpha_X(q)$ in case $Y_t=h(X_t)$  but that the same heredity relation does not hold for the weak dependence coefficients as $\theta$; more tricky arguments as the Proposition 2.1 in \cite{DDLLLP}  are needed.
\end{rmk}

\subsection{Second order behaviour}\label{ss_second order}
For a  stationary dependent sequence with mean $m$,  $(X_i)_{i\in\Z} $, we assume that  partial sums  are asymptotically Gaussian, in such a way that the following convergence holds with $\sigma^2$ defined from \eqref{sigma}
\begin{eqnarray}\label{clt}
\Gamma_{p}&\underset{ p\to\infty}{\to} &  {\cal N}(0,\sigma^2), \qquad\mbox{in distribution},\\
\nonumber  \mbox{with}\qquad\qquad 
\Gamma_{p}&\equiv& \sqrt p\cdot \frac1p\sum_{i=1}^p(X_i-\E X_i).
\end{eqnarray}
\begin{rmk}
	The above assumption holds e.g. in case the condition of \cite{DMR94} or \cite{DD2003} hold. Assumptions considered in the paper are  strong mixing \cite{D1994} or weak dependence conditions, \cite{DL1999}. Many alternative assumptions such as Wu's physical measure of dependence or Mixingale assumptions are also often used.
\end{rmk}
\noindent
As an assumption, let us also assume that there exist constants $c>0$ and $a>1$ such that
\begin{eqnarray}\label{deccov}
|\cov (X_0,X_j)|\le c(|j|+1)^{-a}, \qquad \forall j\in\Z.
\end{eqnarray}
Hence, if  $ a\ge2$ in the assumption \eqref{deccov} we first derive that
\begin{eqnarray}\label{bor2} \E  \Gamma_{p}^2&=& {\cal O}(1), \qquad \mbox{if }\quad a\ge2.
\end{eqnarray}
We now need to  bound $\E \Gamma_{p}^2-\sigma^2$. A simple decomposition yields
\begin{eqnarray*}
	\sigma^2-\E  \Gamma_{p}^2&=& \sum_{|j|>p}\cov (X_0,X_j)+\frac1{p}\sum_{|j|\le p}|j|\cov (X_0,X_j).
\end{eqnarray*}
Then for a suitable constant $\zeta(a)$ only depending on $a$, we can write
\begin{eqnarray}\nonumber
|\sigma^2-\E  \Gamma_{p}^2|&\le& 2c\sum_{j>p}(|j|+1)^{-a}+\frac c{p}\sum_{|j|\le p}(|j|+1)^{1-a}
\le \zeta(a)c \Big(p^{1-a}+\frac 1{p}\Big),\\
\nonumber
&\le& \frac{2\zeta(a)c}p,\qquad  \qquad  \mbox{if }\quad a\ge2 .
\end{eqnarray}
Finally, we have
\begin{eqnarray}\label{borErr} \sigma^2-\E  \Gamma_{p}^2&=& {\cal O}\Big(\frac1p\Big), \qquad \mbox{if }\quad a\ge2.
\end{eqnarray}
\begin{rmk} When $1<a\le 2$, we observe the weaker bound $\sigma^2-\E  \Gamma_{p}^2= {\cal O}\big(p^{1-a}\big)$. More generally, the above bound may be written as $\sigma^2-\E  \Gamma_{p}^2= {\cal O}\big(p^{-\{1\wedge(a-1)\}}\big)$ for each $a>1$. For the sake of simplicity we will assume that $a>2$.
\end{rmk}

\subsection{Moments of approximate Gaussian sums}\label{ss_moment}

Let $I\subset \Z$ be an interval with cardinal $p$, and let $(\widetilde X_i)_{i\in \Z}$ be a stationary sequence of centred random variables (in the current setting we simply write $\widetilde X_i=X_i-m$), we consider in this section 
the behaviour of approximate Gaussian sums of an independent interest. 
\begin{equation}\label{somme}
G_I=\frac1{\sqrt{p}}\sum_{i\in I}\widetilde X_i.
\end{equation}
In Subsection \ref{ss_second order}, we consider the second order behaviour of normalised partial sums $\Gamma_p=G_{[1,p]}$. The present section aims at setting higher order considerations providing asymptotic functional behavior for the process  $(G_{[t+1,t+p]})_{t\in \Z}$ as $p\to\infty$. We will need such  dependence assumptions ensuring that there exist a constant $C$ such that if the  interval $I\subset \Z$ includes less than $p$ values, then:
\begin{eqnarray}\label{mom}
E G_I ^4\le C^4.
\end{eqnarray}
For instance, Lemma 2 and Corollary 2 (with $p=\delta$ in their notation) in \cite{DD2003} allow to state the following result.
\begin{lemma} [\cite{DD2003}]\label{mom4}
	Let $(\widetilde X_i)_{i\in \Z}$ be a stationary centred sequences. Define normalised sum in \eqref{somme} 
	over an interval $I=[a,b]$ with $0\le b-a\le p$, then it exists a constant $C$ such that for each $p$, and for each interval with cardinal 
	less or equal to $p$, so that
	$$ \|G_I\|_\delta\le C.$$ holds if respectively $\alpha$-mixing or $\theta-$weak dependence  hold and
	\begin{eqnarray*}
		\|\widetilde X_0\|_r\sum_{q=1}^\infty q^{\frac{\delta r-2r+1}{r-\delta}}\alpha(q)&<&\infty,\\
		\|\widetilde X_0\|_r\sum_{q=1}^\infty q^{\frac{\delta r-2r+1}{r-\delta}}\theta(q)&<&\infty.
	\end{eqnarray*}
\end{lemma}
\begin{rmk} \label{remsum} The case $\delta=4 $ is of a special interest and conditions to ensure that \eqref{mom} holds 
	become respectively under  the following $\alpha$-mixing or $\theta-$weak dependence :\begin{eqnarray}
	\label{momentalpha}
	\|\widetilde X_0\|_r\sum_{q=1}^\infty q^{\frac{2r+1}{r-4}}\alpha(q)&<&\infty,\\
	\label{momenttheta}
	\|\widetilde X_0\|_r\sum_{q=1}^\infty q^{\frac{2r+1}{r-4}}\theta(q)&<&\infty.
	\end{eqnarray}
	In case  $\alpha(q)={\cal O}(q^{-\alpha})$, $\theta(q)={\cal O}(q^{-\theta})$ satisfy respectively   $\alpha>2\cdot\frac{r-2}{r-4}$ or $\theta>2\cdot\frac{r-2}{r-4}$, then there also exists $\delta>4$ such that the conclusions of Lemma \ref{mom4} still hold.
\end{rmk}

\noindent
Let $I,J$ be two such disjoints sets with cardinal at most equal to $p$, then we will need to prove that
$$\lim_p\cov(G_I^2,G_J^2)=0.$$
First note that 
$$\cov(G_I^2,G_J^2)=\frac1{p^2}\sum_{i,j\in I}\sum_{k,\ell\in J}\cov(\widetilde X_i\widetilde X_j,\widetilde X_k\widetilde X_\ell ).$$
This expression is bounded by using  the following bounds. \\
\rmi First, note that $|\cov(\widetilde X_i\widetilde X_j,\widetilde X_k\widetilde X_\ell )|\le 2\E \widetilde X_0^4$ follows from a systematic application of Cauchy-Schwartz inequality.
Hence, if now $d(\{i,j\},\{k,\ell\})\ge q$ then we have 
\begin{itemize}
	\item Under strong mixing, the covariance inequality in \cite{Rio} ensures that  $$|\cov(\widetilde X_i\widetilde X_j,\widetilde X_k\widetilde X_\ell )|\le 6\| \widetilde X_0\|_r^4\;\alpha^{\frac{r-4}r}(q).$$
	\item Under $\theta$-weak dependence, we use a truncation at a level $M>0$ to be settled later on, i.e.
	$\overline X=\big(\widetilde X\vee(-M)\big)\wedge M$ and
	$\underline X=\widetilde X-\overline X$, then 
	$|\underline  X|\le 2|\widetilde X|\1_{\{|\widetilde X|\ge M\}}$ and
	\begin{eqnarray*}
		|\cov(\widetilde X_i\widetilde X_j,\widetilde X_k\widetilde X_\ell )|&\le& |\cov(\overline X_i\overline X_j,\overline X_k\overline X_\ell )|+ \sum_{u=1}^7A_u\le 2M^2\theta(q)+\sum_{u=1}^7A_u.
	\end{eqnarray*}
	Terms $A_u$, for $u=1,\dots, 7$, are obtained through expansions $\widetilde X=\underline X+\overline X$. Thus each term $A_u$ is a covariance of products including at least one factor $\underline X$. Therefore, Markov inequality, with $\mu=\E|\widetilde X_0|^r$ and with the coefficient $ 112=7\cdot 2^4$, leads to	
	\begin{eqnarray*}
		|\cov(\widetilde X_i\widetilde X_j,\widetilde X_k\widetilde X_\ell ) &\le& 2M^2\theta(q)+112\E| \widetilde X_0|^4\1_{\{|\widetilde X_0|\ge M\}} \le 2M^2\theta(q)+112M^{4-r}\E| \widetilde X_0|^r,\\
		&\le& 4(56\mu)^{\frac{2}{r-2}}\theta^{\frac{r-4}{r-2}}(q), \quad \mbox{with }\quad  M=\Big(\frac {56\mu}{\theta(q)}\Big)^{\frac{1}{r-2}}\\
		&\le& c\theta^{\frac{r-4}{r-2}}(q).
	\end{eqnarray*}
\end{itemize}
Hence now we may come to precise bounds.\\
\rmi  Now, if $d(I,J)\ge q$, then 
\begin{eqnarray*}\label{cov_2}|\cov(G_I^2 ,G_J^2)|&\le& p^2\epsilon_q,
\end{eqnarray*}
with,  for a convenient constant $c>0$, and such that 
$$
\begin{array}{llll}
\epsilon_q&=&6\| \widetilde X_0\|_r^4\ \alpha^{\frac{r-4}r}(q),\qquad\qquad &\mbox{under strong mixing,}
\\
\epsilon_q&=& c\ \theta^{\frac{r-4}{r-2}}(q),\quad \qquad&\mbox{under $\theta$-dependence.}
\end{array}
$$
\begin{lemma}\label{cov2}
	Let $(\widetilde X_i)_{i\in \Z}$ be a stationary centred sequences and define normalised sum \eqref{somme} over  intervals $I,J$ with cardinal less or equal to $p$ and distant at least equal to $q$. Then, if the sequence is either
	strong mixing or  $\theta$-dependent  and  $\|\widetilde X_0\|_r$ for some $r>4$, 
	there exists a constant $c>0$ such that 
	\begin{eqnarray*}\label{cov2theta}
		\label{cov2alpha}
		|\cov(G_I^2 ,G_J^2)|\le c p^2\;\alpha^{\frac{r-4}r}(q), \quad \text{and}\quad 				|\cov(G_I^2 ,G_J^2)|\le c p^2\;\theta^{\frac{r-4}{r-2}}(q),
	\end{eqnarray*}
	respectively, under strong mixing and  $\theta$-dependent conditions.
\end{lemma}
\noindent\rmiii Finally, if $d(I,J)<q$, and $q<p$ then let  $I'\subset I$ be such that distance of $I'$ to $J$ is more than $q$ 
(if $I=[a,b]$ then either $I'=[a,b-q]$ or $I'=[a+q,b]$) we set
$G_I=G+G'$ with
$$ G=\frac1{\sqrt{p}}\sum_{i\in I'}\widetilde X_i.
$$
Then, using Cauchy-Schwartz inequality together with Lemmas  \ref{mom4} and \ref{cov2}, we obtain, with the notation \eqref{cov_2}, that
\begin{eqnarray}\nonumber
|\cov(G_I^2 ,G_J^2)|&\le&|\cov(G^2 ,G_J^2)|+2|\cov(GG' ,G_J^2)|+|\cov(G'^2 ,G_J^2)|,
\\
\nonumber
&\le& p^2\epsilon_q+4(\E G^4)^{\frac14}(\E G'^4)^{\frac14}(\E G_J^4)^{\frac12}+
4(\E G'^4)^{\frac12}(\E G_J^4)^{\frac12},
\\
\nonumber
&\le& p^2\epsilon_q+4C^4\sqrt{\frac{ q}{p}}+4C^4{\frac{ q}{p}},
\\
\label{cov_}
&\le &p^2\epsilon_q+8C^4\sqrt{\frac{ q}{p}}.
\end{eqnarray}
In this way, if conditions \eqref{momentalpha} or  \eqref{momenttheta} hold then $\lim_p p^2\alpha(p)=0$ or  $\lim_p p^2\theta(p)=0$ since $\frac{ 2r+1}{r-4}>2$. This allows to find $q\equiv q(p)\ll p$ such that the right-hand side of \eqref{cov_} tends to zero as $p\to\infty$.

Now, more quantitatively, let us assume that $\epsilon_q\le c' q^{-\kappa}$. Then a choice $q=p^\epsilon$ for some $0<\epsilon=\frac5{2\kappa+1}<1$ gives the following  bound for a constant $c_0>0$,
\begin{eqnarray*}\label{cov^}
|\cov(G_I^2 ,G_J^2)|&\le&c_0p^{-\frac{\kappa-2}{2\kappa+1}}.
\end{eqnarray*}
These bounds all require that $\kappa>2$ which means that if $\alpha(q)={\cal O}(q^{-\alpha})$, $\theta(q)={\cal O}(q^{-\theta})$ we will need respectively $\alpha>2\cdot\frac{r}{r-4}$ or $\theta>2\cdot\frac{r-2}{r-4}$.
Elementary calculations prove that conditions in Lemma \ref{mom4} require exactly  the same assumptions and Remark \ref{remsum} even ensures that $\|G_I\|_4\le C$ whatever $p$ is and the interval $I$ such that $\# I\le p$ in this case.
\begin{lemma}\label{cov3}
	Let $(\widetilde X_i)_{i\in \Z}$ be a stationary centred sequence with $\|\widetilde X_0\|_r<\infty$ for some $r>4$ then defining normalised sum in \eqref{somme} over  intervals $I,J$ with cardinal less or equal to $p$ .  
	\begin{enumerate}
		\item
		Assume that either condition \eqref{momentalpha} or  \eqref{momenttheta} holds then 
		$$\lim_p \sup_{(I,J)\in C(p)}\cov(G_I^2 ,G_J^2)=0,$$
		where the supremum is considered over the collection of intervals $(I,J)$, with  $I,J\subset \Z$, $\#I\le p$,  $\#I\le p$, and $I\cap J=\emptyset$.
		\item
		Assume now that either strong mixing or $\theta$-dependence holds, i.e. $\alpha(q)={\cal O}(q^{-\alpha})$, $\theta(q)={\cal O}(q^{-\theta})$; then in case $\alpha>2\cdot\frac{r}{r-4}$ or $\theta>2\cdot\frac{r-2}{r-4}$, there exists a constant $\gamma$ such that:
		\begin{eqnarray*}\label{cov^}
		|\cov(G_I^2 ,G_J^2)|\le \gamma p^{-a},\quad  a=\left\{\begin{array}{l}\displaystyle
		\frac{\alpha(r-4)-r}{2\alpha(r-4)+r} \\ \displaystyle
		\frac{\theta(r-4)-r}{2\theta(r-4)+r}\end{array}\right., \mbox{ respectively.}
		\end{eqnarray*}
	\end{enumerate}
\end{lemma}
\begin{rmk} If $I=[a,b]$,  $J=[c,d]$ then $0\le b-a\le p$, $0\le d-c\le p$ and $c>b$ in Lemma \ref{cov3}, or $c\ge b+q$ in Lemma \ref{cov2}.
\end{rmk}


\subsection{Dependence properties of expressions of interest}\label{ss2}
In this appendix we will  recall more heredity considerations as extensions to Bernstein blocks variants of Proposition 2.1 in \cite{DDLLLP}.

\noindent\rmi	If the process $(X_i)_{i\in \Z}$ is strong mixing then heredity is simple and 
	$$ \alpha^{U}_n(q ) \vee \alpha^{G}_n(q)\le \alpha((q-1)p_n+1).$$\\
\noindent\rmii If the process $(X_i)_{i\in \Z}$ is $\theta$-weakly dependent, then heredity is more tricky. For this we will use heredity results in Lemma 6 in \cite{BDL} to show that the sequence $(U_{i,n})_{i\in \N}$ is also $\theta_n^U(q)$-weakly dependent. To this end, we first consider the normalised partial sums process $G_{i,n}$. Let $f:\R^u\to\R$ and $h:\R^v\to\R$ be functions  as in \eqref{theta}, then set $S=B_{i_1,n}\cup\cdots \cup B_{i_u,n}$ and  $T=B_{j_1,n}\cup\cdots \cup B_{j_v,n}$. Finally, denote
	\begin{eqnarray*}
		F_n((x_s)_{s\in S})&=&f\Big(\frac1{{\sqrt p_n}}\Big(\sum_{s\in B_{i_\ell,n }}(x_s-m)\Big)_{1\le \ell\le u}\Big),\\
		H_n((x_t)_{t\in T})&=&h\Big(\frac1{{\sqrt p_n}}\Big(\sum_{t\in B_{j_\ell,n }}(x_t-m)\Big)_{1\le \ell\le v}\Big).
	\end{eqnarray*}
	Thus $\# T\le vp_n$, 
	$\|F_n\|_\infty=\|f\|_\infty$ and $\Lip H_n\le \frac1{{\sqrt p_n}}$
	now the distance between the sets of indices $S$ and $T$ is at least $(q-1)p_n$ and we thus calculate 
	\begin{eqnarray*}
		\big|\cov(f(G_{i_1,n},\ldots,G_{i_u,n}), h(G_{j_1,n},\ldots,G_{j_v,n}))\big|&=&
		\big|\cov(F_n((X_s)_{s\in S}), H_n((X_t)_{s\in T})\big|,
		\\
		&\le& vp_n\cdot \frac1{{\sqrt p_n}}\Lip H_n\,\|F_n\|_\infty \theta((q-1)p_n+1),
		\\
		&\le& v{\sqrt p_n}\Lip h\,\|f\|_\infty \theta((q-1)p_n+1),
		\\
		&\le&v\Lip h\,\|f\|_\infty \theta^{G}_n(q).
	\end{eqnarray*}
	Thus, we have
	\begin{eqnarray}\label{thetaG}
	\theta^{G}_n(q)\le {\sqrt p_n} \theta((q-1)p_n+1).
	\end{eqnarray}
	Next, we use Lemma 6 in \cite{BDL} to derive that $\theta_n^U(q)\le C\left(\theta_n^G(q) \right)^{\frac{r-2}{r-1}}$ for a constant $C>0$, 
	if $\E|X_0|^r<\infty $. Finally, from \eqref{thetaG}, it follows that
	\begin{eqnarray*}\label{thetaU}
	\theta_n^U(q)\le C p_n^{\frac{r-2}{2(r-1)}} \theta^{\frac{r-2}{r-1}}((q-1)p_n+1).
	\end{eqnarray*}

\end{document}